\documentclass{amsproc}
\usepackage{amsmath}
\usepackage{enumerate}
\usepackage{amsmath,amsthm,amscd,amssymb}

\usepackage{latexsym}
\usepackage{upref}
\usepackage{verbatim}

\usepackage[mathscr]{eucal}
\usepackage{dsfont}

\usepackage{graphicx}
\usepackage[colorlinks,hyperindex,hypertex]{hyperref}

\newtheorem{theorem}{Theorem}%[section]

\newtheorem{definition}[theorem]{Definition}

\chardef\bslash=`\\ % p. 424, TeXbook

\newcommand{\dA}{{\dot A}}

\newcommand{\dom}{\text{\rm{Dom}}}

\newcommand{\calH}{{\mathcal H}}

\newcommand{\calR}{{\mathcal R}}

\def\bA{{\mathbb A}}      \def\dC{{\mathbb C}}

      \def\dR{{\mathbb R}}

   \def\cH{{\mathcal H}}

\def\RE{{\rm Re\,}}

\def\uphar{{\upharpoonright\,}}

\DeclareMathOperator{\IM}{Im}

\DeclareMathOperator{\Ext}{Ext}

\newcommand{\eval}[2][\right]{\relax
  \ifx#1\right\relax \left.\fi#2#1\rvert}

\begin{document}

\title[Original  Weyl-Titchmarsh functions]{The original  Weyl-Titchmarsh functions and sectorial  Schr\"odinger  L-systems}

\author{S. Belyi}
\address{Department of Mathematics\\ Troy University\\
Troy, AL 36082, USA\\
%URL: {\sf http://spectrum.troy.edu/$\sim$belyi/}
}
\curraddr{}
\email{sbelyi@troy.edu}
%\thanks{Thank you.}

%\author[Makarov]{K. A. Makarov}
%\address{Department of Mathematics\\
% University of Missouri\\
%  Columbia, MO 63211, USA}
%\email{makarovk@missouri.edu}

\author{E. Tsekanovski\u i}
\address{Department of Mathematics\\ Niagara University\\
Lewiston, NY 14109\\ USA}
\email{tsekanov@niagara.edu}
%\thanks{Thank you very much.}

%    General info
\subjclass{Primary 47A10; Secondary 47N50, 81Q10}
\date{DD/MM/2004}

\dedicatory{Dedicated with great pleasure to Seppo Hassi on the occasion of his 60-th birthday}

\keywords{L-system, Schr\"odinger  operator, transfer function, impedance function,  Herglotz-Nevanlinna function, Stieltjes function, Weyl-Titch\-marsh function}

\begin{abstract}
In this paper we study the L-system realizations generated by the original Weyl-Titchmarsh functions  $m_\alpha(z)$ in the case when the minimal symmetric Schr\"o\-dinger  operator in $L_2[\ell,+\infty)$ is non-negative.
We realize functions $(-m_\alpha(z))$ as impe\-dance functions of  Schr\"odinger  L-systems and  derive  necessary and sufficient conditions for  $(-m_\alpha(z))$ to fall into sectorial classes  $S^{\beta_1,\beta_2}$ of Stieltjes functions. Moreover, it is shown that the knowledge of the value  $m_\infty(-0)$ and parameter $\alpha$ allows us to describe the geometric structure of the L-system that realizes  $(-m_\alpha(z))$. Conditions when the main and state space operators  of the L-system  realizing $(-m_\alpha(z))$ have the same or not angle of sectoriality are presented in terms of the parameter $\alpha$. Example that illustrates the obtained results is presented in the end of the paper.

 \end{abstract}

\maketitle

\tableofcontents

\section{Introduction}\label{s1}

%%%%%%%%%%%%%%%%%%%
This paper is a part of an ongoing project studying the realizations of  the {original Weyl-Titchmarsh function} $m_\infty(z)$ and its {linear-fractional transformation} $m_\alpha(z)$ associated with a Schr\"odinger operator in $L_2[\ell,+\infty)$.  In this  project the Herglotz-Nevanlin\-na functions $(-m_\infty(z))$ and $(1/m_\infty(z))$ as well as $(-m_\alpha(z))$ and $(1/m_\alpha(z))$   are being realized as impedance functions of L-systems with a dissipative Schr\"o\-dinger main operator $T_h$, ($\IM h>0$). For the sake of brevity we will refer to these L-systems as \textit{Schr\"odinger L-systems} for the rest of the manuscript.   The formal definition, exposition and discussions of general and Schr\"odinger L-systems are presented in Sections \ref{s2} and  \ref{s4}. We capitalize on the fact that all Schr\"odinger L-systems $\Theta_{\mu,h}$ form a two-parametric family whose members are uniquely defined by a real-valued parameter $\mu$ and a complex boundary value $h$  of the main dissipative operator.

The focus of this paper is set on the case when  the realizing Schr\"odinger L-systems are based on non-negative symmetric Schr\"odinger  operator with $(1,1)$ deficiency indices and have accretive state-space operator. It is known  (see \cite{ABT}) that in this case the impedance functions of such L-systems are Stieltjes. Here we study the situation when the realizing Schr\"odinger L-systems are also sectorial and the  Weyl-Titchmarsh functions $(-m_\alpha(z))$ fall into \textit{sectorial} classes $S^\beta$ and $S^{\beta_1,\beta_2}$ of  Stieltjes functions that are discussed in details in Section \ref{s3}.
Section \ref{s5} provides us with the general realization results (obtained in \cite{BT18}) for the functions $(-m_\infty(z))$,  $(1/m_\infty(z))$, and $(-m_\alpha(z))$. It is shown that $(-m_\infty(z))$,  $(1/m_\infty(z))$, and $(-m_\alpha(z))$ can be realized as the impedance function of  Schr\"odinger L-systems  $\Theta_{0,i}$, $\Theta_{\infty,i}$, and $\Theta_{\tan\alpha, i}$, respectively.

   The main  results of the paper are contained in Section \ref{s7}. Here we apply the realization theorems from Section  \ref{s5} to  Schr\"odinger L-systems that are based on non-negative symmetric Schr\"odinger  operator to obtain additional properties. %In particular, in Theorem \ref{t-14} we describe the cases when the realizing Schr\"odinger L-systems are accretive.
   Utilizing the results presented in Section \ref{s4}, we derive some new features of  Schr\"odinger L-systems $\Theta_{\tan\alpha, i}$ whose impedance functions fall into  particular sectorial classes $S^{\beta_1,\beta_2}$ with $\beta_1$ and $\beta_2$ explicitly described. The results are given in terms of the  parameter $\alpha$ that appears in the definition of the function $m_\alpha(z)$. Moreover, the knowledge of the limit value  $m_\infty(-0)$ and the value of $\alpha$ allows us to find the angle of sectoriality of the main and state-space operators of the realizing L-system.  This, in turn,  leads to connections to Kato's problem about sectorial extension of sectorial forms.

 The paper is concluded with an example that illustrates   main results and concepts. The present work is a further development of the theory of open physical systems conceived by M.~Liv\u sic in \cite{Lv2}.

\section{Preliminaries}\label{s2}

For a pair of Hilbert spaces $\calH_1$, $\calH_2$ we denote by $[\calH_1,\calH_2]$ the set of all bounded linear operators from
$\calH_1$ to $\calH_2$. Let $\dA$ be a closed, densely defined, symmetric operator in a Hilbert space $\calH$ with inner product
$(f,g),f,g\in\calH$. Any non-symmetric operator $T$ in $\cH$ such that $\dA\subset T\subset\dA^*$ is called a \textit{quasi-self-adjoint extension} of $\dA$.

 Consider the rigged Hilbert space (see \cite{Ber}, \cite{ABT})
%\label{107}
$\calH_+\subset\calH\subset\calH_- ,$ where $\calH_+ =\dom(\dA^*)$ and
%\label{108}
\begin{equation}\label{108}
(f,g)_+ =(f,g)+(\dA^* f, \dA^*g),\;\;f,g \in \dom(A^*).
\end{equation}
Let $\calR$ be the \textit{\textrm{Riesz-Berezansky   operator}} $\calR$ (see \cite{Ber}, \cite{ABT}) which maps $\mathcal H_-$ onto $\mathcal H_+$ such
 that   $(f,g)=(f,\calR g)_+$ ($\forall f\in\calH_+$, $g\in\calH_-$) and
 $\|\calR g\|_+=\| g\|_-$.
 Note that
identifying the space conjugate to $\calH_\pm$ with $\calH_\mp$, we
get that if $\bA\in[\calH_+,\calH_-]$, then
$\bA^*\in[\calH_+,\calH_-].$
% \begin{definition}
An operator $\bA\in[\calH_+,\calH_-]$ is called a \textit{self-adjoint
bi-extension} of a symmetric operator $\dA$ if $\bA=\bA^*$ and $\bA
\supset \dA$.
%\end{definition}
Let $\bA$ be a self-adjoint
bi-extension of $\dA$ and let the operator $\hat A$ in $\cH$ be defined as follows:
$$
\dom(\hat A)=\{f\in\cH_+:\bA f\in\cH\}, \quad \hat A=\bA\uphar\dom(\hat A).
$$
The operator $\hat A$ is called a \textit{quasi-kernel} of a self-adjoint bi-extension $\bA$ (see \cite{TSh1}, \cite[Section 2.1]{ABT}).
%According to the von Neumann Theorem (see \cite[Theorem 1.3.1]{ABT}) the domain of $\wh A$, a self-adjoint extension of $\dA$,  can be expressed as
%\begin{equation}\label{DOMHAT}
%\dom(\hat A)=\dom(\dA)\oplus(I+U)\sN_{i},
%\end{equation}
%where von Neumann's parameter $U$ is a $(\cdot)$ (and $(+)$)-isometric operator from $\sN_i$ into $\sN_{-i}$  and $$\sN_{\pm i}=\Ker (\dA^*\mp i I)$$ are the deficiency subspaces of $\dA$.
 A self-adjoint bi-extension $\bA$ of a symmetric operator $\dA$ is called \textit{t-self-adjoint} (see \cite[Definition 4.3.1]{ABT}) if its quasi-kernel $\hat A$ is self-adjoint operator in $\calH$.
An operator $\bA\in[\calH_+,\calH_-]$  is called a \textit{quasi-self-adjoint bi-extension} of an operator $T$ if $\bA\supset T\supset \dA$ and $\bA^*\supset T^*\supset\dA.$  We will be mostly interested in the following type of quasi-self-adjoint bi-extensions.
%\begin{definition}[\cite{ABT}]\label{star_ext}
Let $T$ be a quasi-self-adjoint extension of $\dA$ with nonempty resolvent set $\rho(T)$. A quasi-self-adjoint bi-extension $\bA$ of an operator $T$ is called (see \cite[Definition 3.3.5]{ABT}) a \textit{($*$)-extension } of $T$ if $\RE\bA$ is a
t-self-adjoint bi-extension of $\dA$.
%\end{definition}
In what follows we assume that $\dA$ has deficiency indices $(1,1)$. In this case it is known \cite{ABT} that every  quasi-self-adjoint extension $T$ of $\dA$  admits $(*)$-extensions.
%equal finite deficiency indices and will say that a quasi-self-adjoint extension $T$ of $\dA$ belongs to the \textit{class $\Lambda(\dA)$} if $\rho(T)\ne\emptyset$, $\dom(\dA)=\dom(T)\cap\dom(T^*)$, and hence  $T$ admits $(*)$-extensions.
The description of all $(*)$-extensions via Riesz-Berezansky   operator $\calR$ can be found in \cite[Section 4.3]{ABT}.

Recall that a linear operator $T$ in a Hilbert space $\calH$ is called \textbf{accretive} \cite{Ka} if $\RE(Tf,f)\ge 0$ for all $f\in \dom(T)$.  We call an accretive operator $T$
\textbf{$\beta$-sectorial} \cite{Ka} if there exists a value of $\beta\in(0,\pi/2)$ such that
\begin{equation}\label{e8-29}
    (\cot\beta) |\IM(Tf,f)|\le\,\RE(Tf,f),\qquad f\in\dom(T).
\end{equation}
We say that the angle of sectoriality $\beta$ is \textbf{exact} for a $\beta$-sectorial
operator $T$ if $$\tan\beta=\sup_{f\in\dom(T)}\frac{|\IM(Tf,f)|}{\RE(Tf,f)}.$$
An accretive operator is called \textbf{extremal accretive} if it is not $\beta$-sectorial for any $\beta\in(0,\pi/2)$.
 A $(*)$-extension $\bA$ of $T$ is called \textbf{accretive} if $\RE(\bA f,f)\ge 0$ for all $f\in\cH_+$. This is
equivalent to that the real part $\RE\bA=(\bA+\bA^*)/2$ is a nonnegative t-self-adjoint bi-extension of $\dA$.
%A ($*$)-extensions $\bA$ of an operator $T$  is called \textbf{accumulative} if
%\begin{equation}\label{e7-3-3}
%(\RE\bA f,f)\le (\dA^\ast f,f)+(f,\dA^\ast f),\quad f\in\calH_+.
%\end{equation}

The following definition is a ``lite" version of the definition of L-system given for a scattering L-system with
 one-dimensional input-output space. It is tailored for the case when the symmetric operator of an L-system has deficiency indices $(1,1)$. The general definition of an L-system can be found in \cite[Definition 6.3.4]{ABT} (see also \cite{BHST1} for a non-canonical version).
\begin{definition}
 An array
%\label{141}
\begin{equation}\label{e6-3-2}
\Theta= \begin{pmatrix} \bA&K&\ 1\cr \calH_+ \subset \calH \subset
\calH_-& &\dC\cr \end{pmatrix}%,\quad (\dim E<\infty)
\end{equation}
 is called an \textbf{{L-system}}   if:
\begin{enumerate}
\item[(1)] {$T$ is a dissipative ($\IM(Tf,f)\ge0$, $f\in\dom(T)$) quasi-self-adjoint extension of a symmetric operator $\dA$ with deficiency indices $(1,1)$};
\item[(2)] {$\mathbb  A$ is a   ($\ast $)-extension of  $T$};
%\item[(2)] {$J=J^\ast =J^{-1}\in [E,E],\quad \dim E < \infty $};
\item[(3)] $\IM\bA= KK^*$, where $K\in [\dC,\calH_-]$ and $K^*\in [\calH_+,\dC]$.%, and $\ran(K)=\ran (\IM\bA).$
\end{enumerate}
\end{definition}
%In the definition above   $\varphi_- \in E$ stands for an input vector, $\varphi_+ \in E$ is an output vector, and $x$ is a state space vector in $\calH$.
  Operators $T$ and $\bA$ are called a \textit{main and state-space operators respectively} of the system $\Theta$, and $K$ is  a \textit{channel operator}.
It is easy to see that the operator $\bA$ of the system  \eqref{e6-3-2}  is such that $\IM\bA=(\cdot,\chi)\chi$, $\chi\in\calH_-$ and pick $K c=c\cdot\chi$, $c\in\dC$ (see \cite{ABT}).
  A system $\Theta$ in \eqref{e6-3-2} is called \textit{minimal} if the operator $\dA$ is a prime operator in $\calH$, i.e., there exists no non-trivial reducing invariant subspace of $\calH$ on which it induces a self-adjoint operator. Minimal L-systems of the form \eqref{e6-3-2} with  one-dimensional input-output space were also considered in \cite{BMkT}.

We  associate with an L-system $\Theta$ the  function
\begin{equation}\label{e6-3-3}
W_\Theta (z)=I-2iK^\ast (\mathbb  A-zI)^{-1}K,\quad z\in \rho (T),
\end{equation}
which is called the \textbf{transfer  function} of the L-system $\Theta$. We also consider the  function
\begin{equation}\label{e6-3-5}
V_\Theta (z) = K^\ast (\RE\bA - zI)^{-1} K,
\end{equation}
that is called the
\textbf{impedance function} of an L-system $ \Theta $ of the form (\ref{e6-3-2}).  The transfer function $W_\Theta (z)$ of the L-system $\Theta $ and function $V_\Theta (z)$ of the form (\ref{e6-3-5}) are connected by the following relations valid for $\IM z\ne0$, $z\in\rho(T)$,
\begin{equation*}\label{e6-3-6}
\begin{aligned}
V_\Theta (z) &= i [W_\Theta (z) + I]^{-1} [W_\Theta (z) - I],\\
W_\Theta(z)&=(I+iV_\Theta(z))^{-1}(I-iV_\Theta(z)).
\end{aligned}
\end{equation*}
An L-system $\Theta $ of the form \eqref{e6-3-2} is called an \textbf{accretive L-system} (\cite{BT10}, \cite{DoTs}) if its state-space operator  operator $\bA$ is accretive, that is $\RE(\bA f,f)\ge 0$ for all $f\in \calH_+$.
% and \textbf{accumulative} (\cite{BT11}) if its state-space operator $\bA$ is accumulative, i.e., satisfies \eqref{e7-3-3}. It is easy to see that if an L-system is accumulative, then \eqref{e7-3-3} implies that the operator $\dA$ of the system is non-negative and both operators $T$ and $T^*$ are accretive. We also associate another operator $\ti\bA$ to an accumulative L-system $\Theta$. It is given by
%\begin{equation}\label{e-14}
%    \ti\bA=2\,\RE\dA^*-\bA,
%\end{equation}
%where $\dA^*$ is in $[\calH_+,\calH_-]$. Obviously, $\RE\dA^*\in[\calH_+,\calH_-]$ and $\ti\bA\in[\calH_+,\calH_-]$. Clearly, $\ti\bA$ is a bi-extension of $\dA$ and is accretive if and only if $\bA$ is accumulative. It is also not hard to see that even though $\ti\bA$ is not a ($*$)-extensions  of the operator $T$ but the form $(\ti\bA f,f)$, $f\in\calH_+$ extends the form $(f,T f)$, $f\in\dom(T)$.
 An accretive  L-system  is called  \textbf{sectorial} if the operator  $\bA$ is sectorial, i.e., satisfies \eqref{e8-29} for some $\beta\in(0,\pi/2)$ and all $f\in\cH_+$. %Similarly, an accumulative L-system  is   \textbf{sectorial} if its  operator $\ti\bA$ of the form \eqref{e-14} is sectorial.

\section{Sectorial classes  and  and their realizations}\label{s3}

A scalar function $V(z)$ is called the Herglotz-Nevanlinna function if it is holomorphic on ${\dC \setminus \dR}$, symmetric with respect to the real axis, i.e., $V(z)^*=V(\bar{z})$, $z\in {\dC \setminus \dR}$, and if it satisfies the positivity condition $\IM V(z)\geq 0$,  $z\in \dC_+$.
 The class of all Herglotz-Nevanlinna functions, that can be realized as impedance functions of L-systems, and connections with Weyl-Titchmarsh functions can be found in \cite{ABT}, \cite{BMkT},  \cite{DMTs},  \cite{GT} and references therein.
%Selected class of Herglotz-Nevanlinna functions can be realized as impedance functions of L-systems $\Theta$ of the form \eqref{e6-3-2} whose  definition will be given in Section \ref{s2}.
The following definition  can be found in \cite{KK74}.
%\begin{definition}\label{d8-4-1}
A scalar Herglotz-Nevanlinna function $V(z)$ is a \textit{Stieltjes function} if it is holomorphic in $\Ext[0,+\infty)$ and
\begin{equation}\label{e4-0}
\frac{\IM[zV(z)]}{\IM z}\ge0.
\end{equation}
It is known \cite{KK74} that a Stieltjes function  $V(z)$  admits the following integral representation
\begin{equation}\label{e8-94}
V(z) =\gamma+\int\limits_0^\infty\frac {dG(t)}{t-z},
\end{equation}
where $\gamma\ge0$ and $G(t)$ is a non-decreasing on $[0,+\infty)$  function such that
$\int^\infty_0\frac{dG(t)}{1+t}<\infty.$  We are going to focus on the \textbf{class $S_0(R)$}  (see \cite{BT10}, \cite{DoTs}, \cite{ABT}) of scalar Stieltjes functions such that
 the measure $G(t)$ in  representation \eqref{e8-94} is of unbounded variation. It was shown in \cite{ABT} (see also \cite{BT10}) that such a function $V(z)$ can be realized as the impedance function of an accretive  L-system $\Theta$ of the form \eqref{e6-3-2} with a densely defined symmetric operator if and only if it belongs to the class $S_0(R)$.

Now we are going to consider sectorial subclasses of scalar Stieltjes  functions introduced in \cite{AlTs1}. Let $\beta\in(0,\frac{\pi}{2})$. \textbf{Sectorial subclasses $S^{\beta}$ } of Stieltjes functions are defined as follows:  a scalar Stieltjes function $V(z)$ belongs to $S^{\beta}$ if
\begin{equation}\label{e9-180}
K_{\beta}= \sum_{k,l=1}^n\left[\frac{z_k V(z_k)-\bar z_l V(\bar z_l)}{z_k-\bar z_l}-{{(\cot\beta)}~}V(\bar z_l)V(z_k)\right]h_k\bar h_l\ge0,
\end{equation}
for an arbitrary sequences of complex numbers $\{z_k\}$, ($\IM z_k>0$) and $\{h_k\}$, ($k=1,...,n$).
 For $0<\beta_1< \beta_2 <\frac{\pi}{2}$, we have
\begin{equation*}
S^{ \beta_1}\subset S^{ \beta_2}\subset{S},
\end{equation*}
where $S$ denotes the class of all Stieltjes functions (which corresponds to the case $\beta=\frac{\pi}{2}$). Let $\Theta$ be a minimal L-system of the form \eqref{e6-3-2} with a densely defined non-negative symmetric operator $\dA$. Then (see \cite{ABT}) the impedance function $V_\Theta(z)$ defined by \eqref{e6-3-5} belongs to the class $S^{\beta}$ if and only if the operator $\bA$ of the L-system $\Theta$ is $\beta$-sectorial.

Let $0\le \beta_1<\frac{\pi}{2}$, $0<\beta_2\le \frac{\pi}{2}$, and $\beta_1\le\beta_2$.
We say that a scalar Stieltjes function $V(z)$ belongs to the \textbf{class $S^{\,\beta_1,\beta_2}$} if
\begin{equation}\label{e9-156}
    \tan\beta_1=\lim_{x\to-\infty}V(x),\qquad \tan\beta_2=\lim_{x\to-0}V(x).
\end{equation}
The following connection between the classes $S^{\,\beta}$ and $S^{\,\beta_1,\beta_2}$ can be found in \cite{ABT}.
Let $\Theta$ be an L-system of the form \eqref{e6-3-2} with a densely defined non-negative symmetric operator $\dA$ with deficiency numbers $(1,1).$ Let also $\bA$ be an $\beta$-sectorial $(*)$-extension of $T$. Then the impedance function $V_\Theta(z)$ defined by \eqref{e6-3-5} belongs to the class $S^{\beta_1,\beta_2}$, $\tan\beta_2\le\tan\beta$. Moreover,  the main operator $T$ is $(\beta_2-\beta_1)$-sectorial with the exact  angle of sectoriality $(\beta_2-\beta_1)$.
In the case when  $\beta$  is the exact angle of sectoriality of the operator $T$ we have that $V_\Theta(z)\in S^{0,\beta}$ (see \cite{ABT}).
%\begin{proof}
%According to Theorem \ref{t9-32} the exact angle of sectoriality is given by $\beta_2-\beta_1$, where
%\begin{equation*}%\label{e9-174}
%\tan\beta_1=\lim_{x\to-\infty}V_\Theta(x),\quad
%\tan\beta_2=\lim_{x\to-0}V_\Theta(x).
%\end{equation*}
% It was also shown that $\tan\beta\ge\tan\beta_2$. On the other hand, since in the statement of the current corollary $\beta$ be the exact angle of sectoriality of $T$, then $\beta=\beta_2-\beta_1$ and hence
%$\tan(\beta_2-\beta_1)\ge\tan \beta_2.$
%Therefore, $\beta_1=0$.
%\end{proof}
%\begin{remark}\label{r9-3}
It also follows that under this set of assumptions, the impedance function $V_\Theta(z)$ is such that $\gamma=0$ in representation \eqref{e8-94}.
%\begin{equation*}%\label{e9-175}
%    V_\Theta(z)=\int_0^\infty\frac{dG(t)}{t-z}.
%\end{equation*}
%\end{remark}
%For the remainder of this paper we will need to rely on the following theorem whose proof can be found in \cite{ABT}.
%\begin{theorem}\label{t9-33}

Now let $\Theta$ be an   L-system of the form \eqref{e6-3-2}, where $\bA$ is a $(*)$-extension of $T$ and $\dA$ is a closed densely defined non-negative symmetric operator  with deficiency numbers $(1,1).$ It was proved in \cite{ABT} that if  the impedance function $V_\Theta(z)$ belongs to the class  $S^{\beta_1,\beta_2}$ and $\beta_2\ne\pi/2$, then $\bA$ is $\beta$-sectorial, where
\begin{equation}\label{e9-176}
    \tan\beta=\tan\beta_2+2\sqrt{\tan\beta_1(\tan\beta_2-\tan\beta_1)}.
\end{equation}
%\end{theorem}
 %The next statement was also shown in \cite{ABT} and gives an explicit description of all the functions from the class $S^{\,\beta_1,\beta_2}$ that are realizable as impedance functions of such L-systems that the exact angles of sectoriality of $T$ and $\bA$ coincide.
 Under the above set of conditions on L-system $\Theta$,  it is shown in \cite{ABT} that  $\bA$ is $\beta$-sectorial $(*)$-extension of an $\beta$-sectorial operator $T$ with  the exact angle $\beta\in(0,\pi/2)$ if and only if $V_\Theta(z)\in S^{0,\beta}$.
 % is such that $\gamma=0$ in
 % $$V_\Theta(z)=\int_0^\infty\frac{dG(t)}{t-z}\in S^{0,\beta}.$$
 Moreover, the angle $\beta$ can be found via the formula
 \begin{equation}\label{e9-178-new}
\tan\beta=\int_0^\infty\frac{dG(t)}{t},
 \end{equation}
 where $G(t)$ is the measure from integral representation \eqref{e8-94} of $V_\Theta(z)$.
%%%%%%%%%%%%%%%%%%%%%%%%%%%%% ****************** %%%%%%%%%%%%%%%%%%%%

\section[L-systems with Schr\"odinger operator]{L-systems with Schr\"odinger operator and their impedance functions}\label{s4}

Let $\calH=L_2[\ell,+\infty)$, $\ell\ge0$, and $l(y)=-y^{\prime\prime}+q(x)y$, where $q$ is a real locally summable on  $[\ell,+\infty)$ function. Suppose that the symmetric operator
\begin{equation}
\label{128}
 \left\{ \begin{array}{l}
 \dA y=-y^{\prime\prime}+q(x)y \\
 y(\ell)=y^{\prime}(\ell)=0 \\
 \end{array} \right.
\end{equation}
has deficiency indices (1,1). Let $D^*$ be the set of functions locally absolutely continuous together with their first derivatives such that $l(y) \in L_2[\ell,+\infty)$. Consider $\calH_+=\dom(\dA^*)=D^*$ with the scalar product
$$(y,z)_+=\int_{\ell}^{\infty}\left(y(x)\overline{z(x)}+l(y)\overline{l(z)}
\right)dx,\;\; y,\;z \in D^*.$$ Let $\calH_+ \subset L_2[\ell,+\infty) \subset \calH_-$ be the corresponding triplet of Hilbert spaces. Consider the operators
\begin{equation}\label{131}
 \left\{ \begin{array}{l}
 T_hy=l(y)=-y^{\prime\prime}+q(x)y \\
 hy(\ell)-y^{\prime}(\ell)=0 \\
 \end{array} \right.
           ,\quad  \left\{ \begin{array}{l}
 T^*_hy=l(y)=-y^{\prime\prime}+q(x)y \\
 \overline{h}y(\ell)-y^{\prime}(\ell)=0 \\
 \end{array} \right.,
\end{equation}
where $\IM h>0$.
Let  $\dA$ be a symmetric operator  of the form \eqref{128} with deficiency indices (1,1), generated by the differential operation $l(y)=-y^{\prime\prime}+q(x)y$. Let also $\varphi_k(x,\lambda) (k=1,2)$ be the solutions of the following Cauchy problems:
$$\left\{ \begin{array}{l}
 l(\varphi_1)=\lambda \varphi_1 \\
 \varphi_1(\ell,\lambda)=0 \\
 \varphi'_1(\ell,\lambda)=1 \\
 \end{array} \right., \qquad
\left\{ \begin{array}{l}
 l(\varphi_2)=\lambda \varphi_2 \\
 \varphi_2(\ell,\lambda)=-1 \\
 \varphi'_2(\ell,\lambda)=0 \\
 \end{array} \right.. $$
It is well known \cite{Na68}, \cite{Levitan} that there exists a function $m_\infty(\lambda)$  introduced by H.~Weyl \cite{W}, \cite{W10} for which
$$\varphi(x,\lambda)=\varphi_2(x,\lambda)+m_\infty(\lambda)
\varphi_1(x,\lambda)$$ belongs to $L_2[\ell,+\infty)$. The function $m_\infty(\lambda)$ is not a Herglotz-Nevanlinna function but $(-m_\infty(\lambda))$ and $(1/m_\infty(\lambda))$ are.

%%%%%%%%%%%%%%%%%%%%%%%%%%%
Now we shall construct an L-system based on a non-self-adjoint Schr\"odin\-ger operator $T_h$ with $\IM h>0$.  It  was shown in \cite{ArTs0}, \cite{ABT} that  the set of all ($*$)-extensions of a non-self-adjoint Schr\"odinger operator $T_h$ of the form \eqref{131} in $L_2[\ell,+\infty)$ can be represented in the form
\begin{equation}\label{137}
\begin{split}
&\bA_{\mu, h}\, y=-y^{\prime\prime}+q(x)y-\frac {1}{\mu-h}\,[y^{\prime}(\ell)-
hy(\ell)]\,[\mu \delta (x-\ell)+\delta^{\prime}(x-\ell)], \\
&\bA^*_{\mu, h}\, y=-y^{\prime\prime}+q(x)y-\frac {1}{\mu-\overline h}\,
[y^{\prime}(\ell)-\overline hy(\ell)]\,[\mu \delta
(x-\ell)+\delta^{\prime}(x-\ell)].
\end{split}
\end{equation}
Moreover, the formulas \eqref{137} establish a one-to-one correspondence between the set of all ($*$)-extensions of a Schr\"odinger operator $T_h$ of the form \eqref{131} and all real numbers $\mu \in [-\infty,+\infty]$. One can easily check that the ($*$)-extension $\bA$ in \eqref{137} of the non-self-adjoint dissipative Schr\"odinger operator $T_h$, ($\IM h>0$) of the form \eqref{131} satisfies the condition
\begin{equation*}\label{145}
\IM\bA_{\mu, h}=\frac{\bA_{\mu, h} - \bA^*_{\mu, h}}{2i}=(.,g_{\mu, h})g_{\mu, h},
\end{equation*}
where
\begin{equation}\label{146}
g_{\mu, h}=\frac{(\IM h)^{\frac{1}{2}}}{|\mu - h|}\,[\mu\delta(x-\ell)+\delta^{\prime}(x-\ell)]
\end{equation}
and $\delta(x-\ell), \delta^{\prime}(x-\ell)$ are the delta-function and
its derivative at the point $\ell$, respectively. Furthermore,
\begin{equation*}\label{147}
(y,g_{\mu, h})=\frac{(\IM h)^{\frac{1}{2}}}{|\mu - h|}\ [\mu y(\ell)
-y^{\prime}(\ell)],
\end{equation*}
where $y\in \calH_+$, $g\in \calH_-$, and $\calH_+ \subset L_2[\ell,+\infty) \subset \calH_-$ is the triplet of Hilbert spaces discussed above.

It was also shown in \cite{ABT} that the quasi-kernel $\hat A_\xi$ of $\RE\bA_{\mu, h}$ is given by
\begin{equation}\label{e-31}
  \left\{ \begin{array}{l}
 \hat A_\xi y=-y^{\prime\prime}+q(x)y \\
 y^{\prime}(\ell)=\xi y(\ell) \\
 \end{array} \right.,\quad \textrm{where} \quad  \xi=\frac{\mu\RE h-|h|^2}{\mu-\RE h}.
\end{equation}
Let $E=\dC$, $K_{\mu, h}{c}=cg_{\mu, h}, \;(c\in \dC)$. It is clear that
\begin{equation}\label{148}
K^*_{\mu, h} y=(y,g_{\mu, h}),\quad  y\in \calH_+,
\end{equation}
and $\IM\bA_{\mu, h}=K_{\mu, h}K^*_{\mu, h}.$ Therefore, the array
\begin{equation}\label{149}
\Theta_{\mu, h}= \begin{pmatrix} \bA_{\mu, h}&K_{\mu, h}&1\cr \calH_+ \subset
L_2[\ell,+\infty) \subset \calH_-& &\dC\cr \end{pmatrix},
\end{equation}
is an L-system  with the main operator $T_h$, ($\IM h>0$) of the form \eqref{131},  the state-space operator $\bA_{\mu, h}$ of the form \eqref{137}, and  with the channel operator $K_{\mu, h}$ of the form \eqref{148}.
It was established in \cite{ArTs0}, \cite{ABT} that the transfer and impedance functions of $\Theta_{\mu, h}$ are
\begin{equation}\label{150}
W_{\Theta_{\mu, h}}(z)= \frac{\mu -h}{\mu - \overline h}\,\,
\frac{m_\infty(z)+ \overline h}{m_\infty(z)+h},
\end{equation}
and
\begin{equation}\label{1501}
V_{\Theta_{\mu, h}}(z)=\frac{\left(m_\infty(z)+\mu\right)\IM h}{\left(\mu-\RE h\right)m_\infty(z)+\mu\RE h-|h|^2}.
\end{equation}
It was shown in \cite[Section 10.2]{ABT} that if the parameters $\mu$ and $\xi$ are related via \eqref{e-31}, then the two L-systems $\Theta_{\mu, h}$ and $\Theta_{\xi, h}$ of the form \eqref{149} have the following property
\begin{equation}\label{e-35-mu-xi}
    W_{\Theta_{\mu, h}}(z)=-W_{\Theta_{\xi, h}}(z),\; V_{\Theta_{\mu, h}}(z)=-\frac{1}{V_{\Theta_{\xi, h}}(z)},\; \textrm{where} \quad \xi=\frac{\mu\RE h-|h|^2}{\mu-\RE h}.
\end{equation}

\section{Realizations of $-m_\infty(z)$, $1/m_\infty(z)$ and $m_\alpha(z)$.}\label{s5}

It is known \cite{Levitan}, \cite{Na68} that the original  Weyl-Titchmarsh function $m_\infty(z)$ has a property that $(-m_\infty(z))$ is a Herglotz-Nevanlinna function. The question whether $(-m_\infty(z))$ can be realized as the impedance function of a Schr\"odinger L-system is answered in the following theorem that was proved in \cite{BT18}.

\begin{theorem}[\cite{BT18}]\label{t-6}%
Let $\dA$ be a  symmetric Schr\"odinger operator of the form \eqref{128} with deficiency indices $(1, 1)$ and locally summable potential in $\calH=L^2[\ell, \infty).$ If $m_\infty(z)$ is the  Weyl-Titchmarsh function of $\dA$, then the Herglotz-Nevanlinna function $(-m_\infty(z))$ can be realized as the impedance function of  a Schr\"odinger L-system $\Theta_{\mu, h}$ of the form \eqref{149} with $\mu=0$ and $h=i$.
%\begin{equation}\label{e-35-h-mu}
%    \mu=0\quad \textrm{and}\quad h=i.
%\end{equation}

Conversely, let $\Theta_{\mu, h}$ be  a Schr\"odinger L-system of the form \eqref{149} with the symmetric operator $\dA$ such that $V_{\Theta_{\mu, h}}(z)=-m_\infty(z),$ for all $z\in\dC_\pm$ and $\mu\in\mathbb R\cup\{\infty\}$. Then the parameters $\mu$ and $h$ defining $\Theta_{\mu, h}$ are such that $\mu=0$ and $h=i$.
\end{theorem}

A similar result for the function $1/m_\infty(z)$ was also proved in \cite{BT18}.
\begin{theorem}[\cite{BT18}]\label{t-7}%
Let $\dA$ be a  symmetric Schr\"odinger operator of the form \eqref{128} with deficiency indices $(1, 1)$ and locally summable potential in $\calH=L^2[\ell, \infty).$ If $m_\infty(z)$ is the  Weyl-Titchmarsh function of $\dA$, then the Herglotz-Nevanlinna function $(1/m_\infty(z))$ can be realized as the impedance function of  a Schr\"odinger L-system $\Theta_{\mu, h}$ of the form \eqref{149} with $\mu=\infty$ and $h=i$.
%\begin{equation}\label{e-42-h-mu}
%    \mu=\infty\quad \textrm{and}\quad h=i.
%\end{equation}

Conversely, let $\Theta_{\mu, h}$ be  a Schr\"odinger L-system of the form \eqref{149} with the symmetric operator $\dA$ such that $V_{\Theta_{\mu, h}}(z)=\frac{1}{m_\infty(z)},$ for all $z\in\dC_\pm$ and $\mu\in\mathbb R\cup\{\infty\}$. Then the parameters $\mu$ and $h$ defining $\Theta_{\mu, h}$ are such that $\mu=\infty$ and $h=i$.
\end{theorem}
We note that both L-systems $\Theta_{0,i}$  and $\Theta_{\infty,i}$ obtained in Theorems \ref{t-6} and \ref{t-7} share the same main operator
\begin{equation}\label{e-56-T}
    \left\{ \begin{array}{l}
 T_{i}\, y=-y^{\prime\prime}+q(x)y \\
 y'(\ell)=i\,y(\ell) \\
 \end{array} \right..
\end{equation}

Now we recall the definition of Weyl-Titchmarsh functions $m_\alpha(z)$.  Let  $\dA$ be a symmetric operator  of the form \eqref{128} with deficiency indices (1,1), generated by the differential operation $l(y)=-y^{\prime\prime}+q(x)y$. Let also $\varphi_\alpha(x,{z})$ and $\theta_\alpha(x,{z})$ be the solutions of the following Cauchy problems:
$$\left\{ \begin{array}{l}
 l(\varphi_\alpha)={z} \varphi_\alpha \\
 \varphi_\alpha(\ell,{z})=\sin\alpha \\
 \varphi'_\alpha(\ell,{z})=-\cos\alpha \\
 \end{array} \right., \qquad
\left\{ \begin{array}{l}
 l(\theta_\alpha)={z} \theta_\alpha \\
 \theta_\alpha(\ell,{z})=\cos\alpha \\
 \theta'_\alpha(\ell,{z})=\sin\alpha \\
 \end{array} \right.. $$
It is  known \cite{DanLev90}, \cite{Na68}, \cite{Ti62} that there exists an analytic in $\dC_\pm$  function $m_\alpha({z})$  for which
\begin{equation}\label{e-62-psi}
\psi(x,{z})=\theta_\alpha(x,{z})+m_\alpha({z})\varphi_\alpha(x,{z})
\end{equation}
belongs to $L_2[\ell,+\infty)$. It is easy to see that if $\alpha=\pi$, then $m_\pi({z})=m_\infty({z})$. The functions $ m_\alpha({z})$ and $m_\infty(z)$ are connected (see \cite{DanLev90}, \cite{Ti62}) by
\begin{equation}\label{e-59-LFT}
    m_\alpha({z})=\frac{\sin\alpha+m_\infty({z})\cos\alpha}{\cos\alpha-m_\infty({z})\sin\alpha}.
\end{equation}
We know \cite{Na68}, \cite{Ti62} that for any real $\alpha$ the function $-m_\alpha({z})$ is a Herglotz-Nevanlinna function. Also, modifying \eqref{e-59-LFT} slightly we obtain
\begin{equation}\label{e-61-Don}
    -m_\alpha(z)=\frac{\sin\alpha+m_\infty(z)\cos\alpha}{-\cos\alpha+m_\infty(z)\sin\alpha}
    =\frac{\cos\alpha+\frac{1}{m_\infty(z)}\sin\alpha}{\sin\alpha-\frac{1}{m_\infty(z)}\cos\alpha}.
\end{equation}
The following realization theorem (see \cite{BT18}) for Herglotz-Nevanlinna functions $-m_\alpha(z)$  is similar to Theorem \ref{t-6}.% and \ref{t-7}.
\begin{theorem}[\cite{BT18}]\label{t-8}%
Let $\dA$ be a  symmetric Schr\"odinger operator of the form \eqref{128} with deficiency indices $(1, 1)$ and locally summable potential in $\calH=L^2[\ell, \infty).$ If $m_\alpha(z)$ is the  function of $\dA$ described in \eqref{e-62-psi}, then the Herglotz-Nevanlinna function $(-m_\alpha(z))$ can be realized as the impedance function of  a Schr\"odinger L-system $\Theta_{\mu, h}$ of the form \eqref{149} with
\begin{equation}\label{e-62-h-mu}
    \mu=\tan\alpha\quad \textrm{and}\quad h=i.
\end{equation}

Conversely, let $\Theta_{\mu, h}$ be  a Schr\"odinger L-system of the form \eqref{149} with the symmetric operator $\dA$ such that $$V_{\Theta_{\mu, h}}(z)=-m_\alpha(z),$$ for all $z\in\dC_\pm$ and $\mu\in\mathbb R\cup\{\infty\}$. Then the parameters $\mu$ and $h$ defining $\Theta_{\mu, h}$ are given by \eqref{e-62-h-mu}, i.e., $\mu=\tan\alpha$ and $h=i$.
\end{theorem}

We note that when $\alpha=\pi$ we obtain $\mu_\alpha=0$, $m_\pi(z)=m_\infty(z)$, and the realizing Schr\"odinger L-system $\Theta_{0,i}$ is thoroughly described  in \cite[Section 5]{BT18}. If $\alpha=\pi/2$, then  we get $\mu_\alpha=\infty$, $-m_\alpha(z)=1/m_\infty(z)$, and the realizing Schr\"odinger L-system is $\Theta_{\infty,i}$ (see \cite[Section 5]{BT18}). Assuming that $\alpha\in(0,\pi]$ and  neither $\alpha=\pi$ nor $\alpha=\pi/2$ we give the description of a Schr\"odinger L-system $\Theta_{\mu_\alpha,i}$ realizing $-m_\alpha(z)$ as follows.
\begin{equation}\label{e-64-sys}
    \Theta_{\tan\alpha, i}= \begin{pmatrix} \bA_{\tan\alpha, i}&K_{\tan\alpha, i}&1\cr \calH_+ \subset
L_2[\ell,+\infty) \subset \calH_-& &\dC\cr \end{pmatrix},
\end{equation}
where
\begin{equation}\label{e-65-star}
\begin{split}
&\bA_{\tan\alpha,i}\, y=l(y)-\frac{1}{\tan\alpha-i}[y^{\prime}(\ell)-iy(\ell)][(\tan\alpha)\delta(x-\ell)+\delta^{\prime}(x-\ell)], \\
&\bA^*_{\tan\alpha,i}\, y=l(y)-\frac{1}{\tan\alpha+i}\,[y^{\prime}(\ell)+iy(\ell)][(\tan\alpha)\delta(x-\ell)+\delta^{\prime}(x-\ell)],
\end{split}
\end{equation}
$K_{\tan\alpha, i}\,{c}=c\,g_{\tan\alpha, i}$, $(c\in \dC)$ and
\begin{equation}\label{e-66-g}
g_{\tan\alpha, i}=(\tan\alpha)\delta(x-\ell)+\delta^{\prime}(x-\ell).
\end{equation}
Also,
\begin{equation}\label{e-71-VW}
\begin{aligned}
    V_{\Theta_{\tan\alpha, i}}(z)&=-m_\alpha(z)\\
     W_{\Theta_{\tan\alpha, i}}(z)&=\frac{\tan\alpha-i}{\tan\alpha+i}\cdot\frac{m_\infty(z)-i}{m_\infty(z)+i}=(-e^{2\alpha i})\,\frac{m_\infty(z)-i}{m_\infty(z)+i}.
    \end{aligned}
\end{equation}

The realization theorem for Herglotz-Nevanlinna functions $1/m_\alpha(z)$ is similar to Theorem \ref{t-7} and can be found in \cite{BT18}.

%%%%%%%%%%%%%%%%%%%%%%%%%

\section{Non-negative Schr\"odinger  operator and sectorial L-systems}\label{s7}

Now let us assume that  $\dA$ is a non-negative (i.e., $(\dA f,f) \geq 0$ for all $f \in \dom(\dA)$) symmetric operator  of the form \eqref{128} with deficiency indices (1,1), generated by the differential operation $l(y)=-y^{\prime\prime}+q(x)y$. The following theorem takes place.
\begin{theorem}[\cite{T87},  \cite{Ts81}, \cite{Ts80}]\label{t-10}
Let $\dA$ be a nonnegative symmetric Schr\"odinger operator of the form \eqref{128} with deficiency indices $(1, 1)$ and locally summable potential in $\calH=L^2[\ell,\infty).$ Consider operator $T_h$ of the form \eqref{131}.  Then
 \begin{enumerate}
\item operator $\dA$ has more than one non-negative self-adjoint extension, i.e., the Friedrichs extension $A_F$ and the Kre\u{\i}n-von Neumann extension $A_K$ do not coincide, if and only if $m_{\infty}(-0)<\infty$;
 \item operator $T_h$, ($h=\bar h$) coincides with the Kre\u{\i}n-von Neumann extension $A_K$ if and  only if $h=-m_{\infty}(-0)$;
\item operator $T_h$ is accretive if and only if
\begin{equation}\label{138}
\RE h\geq-m_\infty(-0);
\end{equation}
\item operator $T_h$, ($h\ne\bar h$) is $\beta$-sectorial if and only if  $\RE h >-m_{\infty}(-0)$ holds;
\item operator $T_h$, ($h\ne\bar h$) is accretive but not $\beta$-sectorial for any $\beta\in (0, \frac{\pi}{2})$ if and only if $\RE h=-m_{\infty}(-0)$ \item If $T_h, (\IM h>0)$ is $\beta$-sectorial,
then the exact angle $\beta$ can be calculated via
\begin{equation}\label{e10-45}
\tan\beta=\frac{\IM h}{\RE h+m_{\infty}(-0)}.
\end{equation}
\end{enumerate}
\end{theorem}
For the remainder of this paper we assume that $m_{\infty}(-0)<\infty$. Then according to Theorem \ref{t-10} above (see also \cite{AT2009}, \cite{Ts2}, \cite{Ts80}) we have the existence of the operator $T_h$, ($\IM h>0$) that is accretive and/or sectorial.
It  was shown in \cite{ABT} that if $T_h \;(\IM h>0)$ is an accretive Schr\"odinger operator of the form \eqref{131}, then for all real $\mu$ satisfying the following inequality
\begin{equation}\label{151}
\mu \geq \frac {(\IM h)^2}{m_\infty(-0)+\RE h}+\RE h,
\end{equation}
formulas \eqref{137} define the set of all accretive $(*)$-extensions $\bA_{\mu,h}$ of the operator $T_h$. Moreover, an accretive $(*)$-extensions $\bA_{\mu,h}$ of a sectorial  operator $T_h$ with exact angle of sectoriality $\beta\in(0,\pi/2)$ also preserves the same exact angle of sectoriality if and only if $\mu=+\infty$ in \eqref{137} (see \cite[Theorem 3]{BT-15}). Also, $\bA_{\mu,h}$ is accretive but not  $\beta$-sectorial for any $\beta\in(0,\pi/2)$ ($*$)-extension of $T_h$
 if and only if in \eqref{137}
\begin{equation}\label{e10-134}
\mu=\frac{(\IM h)^2}{m_\infty(-0)+\RE h}+\RE h,
\end{equation}
 (see \cite[Theorem 4]{BT-15}).
An accretive operator $T_h$ has a unique accretive $(*)$-extension $\bA_{\infty,h}$ if and only if
$\RE h=-m_\infty(-0).$
In this case this unique $(*)$-extension has the form
\begin{equation}\label{153}
\begin{aligned}
&\bA_{\infty,h} y=-y^{\prime\prime}+q(x)y+[hy(\ell)-y^{\prime}(\ell)]\,\delta(x-\ell), \\
&\bA^*_{\infty,h} y=-y^{\prime\prime}+q(x)y+[\overline h
y(\ell)-y^{\prime}(\ell)]\,\delta(x-\ell).
\end{aligned}
\end{equation}
%Now we will see how the additional requirement of non-negativity affects the realization of functions $-m_\infty(z)$ and $1/m_\infty(z)$.
%\begin{theorem}\label{t-11}%
%Let $\dA$ be a non-negative symmetric Schr\"odinger operator of the form \eqref{128} with deficiency indices $(1, 1)$ and locally summable potential in $\calH=L^2[\ell, \infty).$ If $m_\infty(z)$ is the  Weyl-Titchmarsh function of $\dA$, then the L-system $\Theta_{0, i}$ of the form \eqref{e-38-sys} realizing the  function $(-m_\infty(z))$ is never accretive. The L-system $\Theta_{\infty, i}$ of the form \eqref{e-49-sys} realizing the  function $1/m_\infty(z)$ is accretive if and only if $m_\infty(-0)\ge0$.
% \end{theorem}

Now we are going to turn to functions $m_\alpha(z)$ described by \eqref{e-62-psi}-\eqref{e-59-LFT} and associated with the non-negative operator $\dA$ above. We need to see how the parameter $\alpha$ in the definition of $m_\alpha(z)$ affects the  L-system realizing $(-m_\alpha(z))$. This  question was answered in \cite[Theorem 6.3]{BT18}. It tells us that if the non-negative symmetric Schr\"odinger operator is such that $m_{\infty}(-0)\ge0$, then the L-system $\Theta_{\tan\alpha, i}$ of the form \eqref{e-64-sys} realizing the  function $(-m_\alpha(z))$ is accretive if and only if
\begin{equation}\label{e-78-angles}
\tan\alpha\ge \frac{1}{m_{\infty}(-0)}.
\end{equation}
%\begin{theorem}[\cite{BT18}]\label{t-12}%
%Let $\dA$ be a non-negative symmetric Schr\"odinger operator of the form \eqref{128} with deficiency indices $(1, 1)$ and locally summable potential in $\calH=L^2[\ell, \infty)$ and such that $m_{\infty}(-0)\ge0$. If $m_\alpha(z)$ is  described by \eqref{e-62-psi}-\eqref{e-59-LFT}, then the L-system $\Theta_{\tan\alpha, i}$ of the form \eqref{e-64-sys} realizing the  function $(-m_\alpha(z))$ is accretive if and only if
%\begin{equation}\label{e-78-angles}
%\tan\alpha\ge \frac{1}{m_{\infty}(-0)}.
%\end{equation}
% \end{theorem}
Note that if $m_\infty(-0)=0$ in \eqref{e-78-angles}, then $\alpha=\pi/2$ and $-m_{\frac{\pi}{2}}(z)={1}/{m_\infty(z)}$. Also, from \cite[Theorem 6.2]{BT18} we know that  if $m_\infty(-0)\ge0$, then ${1}/{m_\infty(z)}$ is realized by an accretive system $\Theta_{\infty, i}$.% of the form \eqref{e-49-sys}. %The system $\Theta_{\infty, i}$ is also extremal

Now once we established a criteria for an L-system realizing $(-m_\alpha(z))$ to be accretive, we can look into more of its properties.
There are two choices for an accretive L-system $\Theta_{\tan\alpha, i}$: it is either (1) \textit{accretive sectorial} or (2) \textit{accretive extremal}. In the case (1) we have that $\bA_{\tan\alpha, i}$ of the form \eqref{e-65-star} is $\beta_1$-sectorial with some angle of sectoriality $\beta_1$ that can only exceed the exact angle of sectoriality $\beta$ of $T_i$. In the case (2) the state-space operator $\bA_{\tan\alpha, i}$ is extremal (not sectorial for any $\beta\in(0,\pi/2)$) and is a $(*)$-extension of $T_i$ that itself can be either $\beta$-sectorial or extremal.  These possibilities were described in details in \cite[Theorem 6.4]{BT18}. In particular, it was shown that for the accretive L-system
 $\Theta_{\tan\alpha, i}$   realizing the  function $(-m_\alpha(z))$  the following is true:
\begin{enumerate}
  \item if $m_{\infty}(-0)=0$, then there is only one accretive L-system $\Theta_{\infty,i}$ realizing $(-m_\alpha(z))$. This L-system is extremal and its main operator $T_i$ is extremal as well.
  \item if $m_\infty(-0)>0$, then $T_i$ is $\beta$-sectorial for $\beta\in(0,\pi/2)$ and
    \begin{enumerate}
      \item if $\tan\alpha={1}/{m_\infty(-0)}$, then $\Theta_{\tan\alpha, i}$ is extremal;
      \item if $\frac{1}{m_\infty(-0)}<\tan\alpha<+\infty$, then $\Theta_{\tan\alpha, i}$ is $\beta_1$-sectorial with $\beta_1>\beta$;
      \item if $\tan\alpha=+\infty$, then $\Theta_{\infty, i}$ is $\beta$-sectorial.
    \end{enumerate}
  \end{enumerate}
\begin{figure}
  \begin{center}
  \includegraphics[width=110mm]{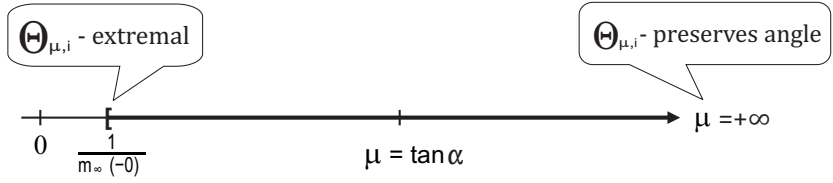}
  \caption{Accretive L-systems $\Theta_{\mu, i}$}\label{fig-1}
  \end{center}
\end{figure}
Figure \ref{fig-1} above describes the dependence of the properties of realizing $(-m_\alpha(z))$ L-systems on the value of $\mu$ and hence $\alpha$. The bold part of the real line depicts values of $\mu=\tan\alpha$ that produce accretive L-systems $\Theta_{\mu, i}$.

Additional analytic properties of the functions $(-m_\infty(z))$, $1/m_\infty(z)$, and \break $(-m_\alpha(z))$ were described in \cite[Theorem 6.5]{BT18}. It was proved there that under the current set of assumptions we have:
 \begin{enumerate}
   \item the  function $1/m_\infty(z)$ is Stieltjes if and only if $m_\infty(-0)\ge0$;
   \item the  function $(-m_\infty(z))$ is never Stieltjes;\footnote{It will be shown in an upcoming paper that if $m_\infty(-0)\ge0$, then the  function $(-m_\infty(z))$ is actually inverse Stieltjes.}
   \item the  function $(-m_\alpha(z))$ given by \eqref{e-59-LFT} is Stieltjes if and only if $$0<\frac{1}{m_\infty(-0)}\le\tan\alpha.$$
 \end{enumerate}

%\section{Sectorial L-systems with Schr\"odinger operator}\label{s6}
Now we are going to turn to the case when our realizing L-system $\Theta_{\tan\alpha, i}$ is accretive sectorial.
To begin with let  $\Theta$ be an L-system of the form \eqref{149}, where $\bA$ is a ($*$)-extension \eqref{137} of the accretive Schr\"odinger operator $T_h$. Here we summarize and list some known facts about possible accretivity and  sectoriality of $\Theta$.
\begin{itemize}
 \item The operator $\bA_{\mu,h}$ of $\Theta_{\mu,h}$ is accretive if and only if \eqref{151} holds (see  \cite{ABT}).
  \item According to Theorem \ref{t-6} if an accretive  operator $T_h$, ($\IM h>0$)  is $\beta$-sectorial, then \eqref{e10-45} holds. Conversely, if $h$, ($\IM h>0$) is such that $\RE h>-m_\infty(-0),$ then operator $T_h$ of the form \eqref{131} is $\beta$-sectorial and $\beta$ is determined by \eqref{e10-45}.
  \item $T_h$  is accretive but not $\beta$-sectorial for any $\beta\in(0,\pi/2)$ if and only if $\RE h=-m_\infty(-0)$.

  \item If $\Theta_{\mu,h}$ is such that $\mu=+\infty$, then $V_{\Theta_{\infty,h}}(z)$ belongs to the class $S^{0,\beta}$. In the case when  $\mu\ne+\infty$ we have $V_{\Theta_{\mu,h}}(z)\in S^{\beta_1,\beta_2}$ (see \cite{B2011}).
  \item The operator  $\bA_{\mu,h}$ is a $\beta$-sectorial ($*$)-extension of $T_h$ (with the same angle of sectoriality) if and only if $\mu=+\infty$ in \eqref{137} (see \cite{ABT}, \cite{BT-15}).
  \item   If $T_h$ is $\beta$-sectorial with the exact angle of sectoriality $\beta$, then it admits only one $\beta$-sectorial ($*$)-extension $\bA_{\mu,h}$ with the same angle of sectoriality $\beta$. Consequently, $\mu=+\infty$ and $\bA_{\mu,h}=\bA_{\infty,h}$ has the form \eqref{153}.
  \item  A ($*$)-extension $\bA_{\mu,h}$ of $T_h$ is accretive but not  $\beta$-sectorial for any $\beta\in(0,\pi/2)$   if and only if the value of  $\mu$ in \eqref{137} is given by \eqref{e10-134}.
\end{itemize}
Note that it follows from the above that any $\beta$-sectorial operator $T_h$ with the exact angle of sectoriality $\beta\in(0,\pi/2)$ admits only one accretive ($*$)-extension $\bA_{\mu,h}$ that is not  $\beta$-sectorial for any $\beta\in(0,\pi/2)$. This extension takes form \eqref{137} with $\mu$  given by \eqref{e10-134}.

%%%%%%%%%%%%%%%%%
Now let us consider a function $(-m_\alpha(z))$ and Schr\"odinger L-system $\Theta_{\tan\alpha, i}$ of the form \eqref{e-64-sys} that realizes it. According to \cite[Theorem 6.4-6.5]{BT18} this L-system $\Theta_{\tan\alpha, i}$ is sectorial if and only if
\begin{equation}\label{e-45}
\tan\alpha>\frac{1}{m_{\infty}(-0)}.
\end{equation}
If we assume that L-system $\Theta_{\tan\alpha, i}$ is $\beta$-sectorial, then its impedance function \break $V_{\Theta_{\tan\alpha, i}}(z)=-m_\alpha(z)$ belongs to certain sectorial classes discussed in Section \ref{s3}. Namely, $(-m_\alpha(z))\in S^\beta$. The following theorem provides more refined properties of $(-m_\alpha(z))$ for this case.
\begin{theorem}\label{t-15}
Let $\Theta_{\tan\alpha, i}$ be the accretive L-system of the form \eqref{e-64-sys} realizing the  function  $(-m_\alpha(z))$  associated with the non-negative operator $\dA$. Let also $\bA_{\tan\alpha,i}$ be a $\beta$-sectorial $(*)$-extension of $T_i$ defined by \eqref{e-56-T}. Then the  function $(-m_\alpha(z))$  belongs to the class $S^{\beta_1,\beta_2}$, $\tan\beta_2\le\tan\beta$,
\begin{equation}\label{e-46}
    \tan\beta_1=\cot\alpha,%\frac{\tan\alpha+m_\infty(-\infty)}{(\tan\alpha) m_\infty(-\infty)-1},
\end{equation}
and
\begin{equation}\label{e-47}
    \tan\beta_2=\frac{\tan\alpha+m_\infty(-0)}{(\tan\alpha) m_\infty(-0)-1}.
\end{equation}
Moreover, the operator $T_i$ is $(\beta_2-\beta_1)$-sectorial with the exact  angle of sectoriality $(\beta_2-\beta_1)$.
 \end{theorem}
 \begin{proof}
It is given that $\Theta_{\tan\alpha, i}$ is $\beta$-sectorial and hence \eqref{e-45} holds. For further convenience we re-write $(-m_\alpha(z))$ as
\begin{equation}\label{e-48}
    -m_\alpha(z)=\frac{\sin\alpha+m_\infty(z)\cos\alpha}{-\cos\alpha+m_\infty(z)\sin\alpha}
    =\frac{\tan\alpha+m_\infty(z)}{(\tan\alpha) m_\infty(z)-1}.
\end{equation}
Since under our assumption  $\Theta_{\tan\alpha, i}$ is $\beta$-sectorial, then its impedance function $V_{\Theta_{\tan\alpha, i}}(z)=-m_\alpha(z)$ belongs to certain sectorial classes discussed in Section \ref{s3}. Namely, $-m_\alpha(z)\in S^\beta$ and $-m_\alpha(z)\in S^{\beta_1,\beta_2}$. In order to describe $\beta_1$ we take into account (see \cite[Section 10.3]{ABT}) that $\lim_{x\to-\infty} m_\infty(x)=+\infty$ to obtain
\begin{equation*}%\label{e-47}
    \begin{aligned}
    \tan\beta_1&=\lim_{x\to-\infty}(-m_\alpha(x))=\frac{\tan\alpha+m_\infty(-\infty)}{(\tan\alpha) m_\infty(-\infty)-1}=\frac{\frac{\tan\alpha}{m_\infty(-\infty)}+1}{\tan\alpha -\frac{1}{m_\infty(-\infty)}}\\
    &=\frac{1}{\tan\alpha}=\cot\alpha.
    \end{aligned}
\end{equation*}
In order to get $\beta_2$ we simply pass to the limit in \eqref{e-48}
\begin{equation*}%\label{e-47}
    \tan\beta_2=\lim_{x\to-0}(-m_\alpha(x))=\frac{\tan\alpha+m_\infty(-0)}{(\tan\alpha) m_\infty(-0)-1}.
\end{equation*}
 The above confirms \eqref{e-46} and \eqref{e-47}. In order to show the rest, we apply \cite[Theorem 9.8.4]{ABT}. This theorem states that if  $\bA$ is a $\beta$-sectorial $(*)$-extension of a main operator $T$ of an L-system $\Theta$, then the impedance function $V_\Theta(z)$ belongs to the class $S^{\beta_1,\beta_2}$, $\tan\beta_2\le\tan\beta$, and  $T$ is $(\beta_2-\beta_1)$-sectorial with the exact  angle of sectoriality $(\beta_2-\beta_1)$. It can also be checked directly that formulas \eqref{e-46} and \eqref{e-47} (under condition \eqref{e-45}) imply $0<\beta_2-\beta_1<\pi/2$ and hence the definition of $(\beta_2-\beta_1)$-sectoriality applies correctly.
 \end{proof}

\begin{figure}
  \begin{center}
  \includegraphics[width=90mm]{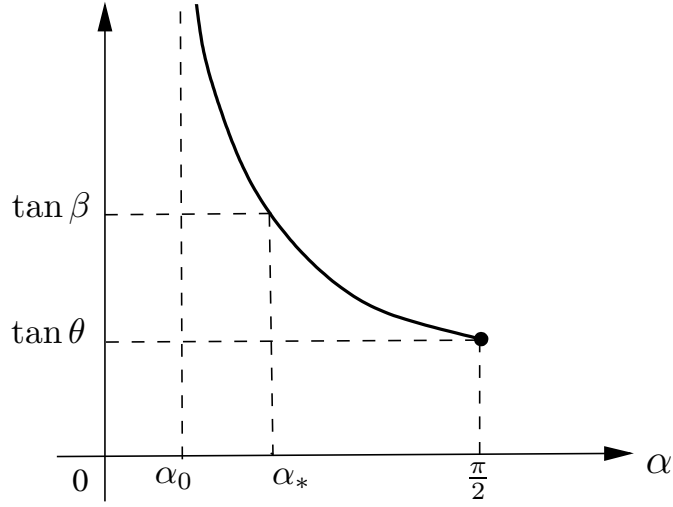}
  \caption{Angle of sectoriality $\beta$. Here $\alpha_0=\arctan\big(\frac{1}{m_{\infty}(-0)}\big)$.}\label{fig-2}
  \end{center}
\end{figure}
%%%%%%%%%%%%%%%%%

Now we state and prove the following.
\begin{theorem}\label{t-16}
Let $\Theta_{\tan\alpha, i}$ be an accretive  L-system of the form \eqref{e-64-sys} that realizes $(-m_\alpha(z))$, where $\bA_{\tan\alpha,i}$ is a ($*$)-extension of a $\theta$-sectorial operator $T_i$ with exact angle of sectoriality $\theta$. Let also  $\alpha_*\in\left(\arctan\big(\frac{1}{m_{\infty}(-0)}\big),\frac{\pi}{2}\right)$ be a fixed value that defines  $\bA_{\tan\alpha_*,i}$ via \eqref{137},   and  $(-m_\alpha(z))\in S^{\beta_1,\beta_2}$.
Then a ($*$)-extension $\bA_{\tan\alpha,i}$ of $T_i$ is $\beta$-sectorial for any $\alpha\in [\alpha_*,\pi/2)$ with
\begin{equation}\label{e6-6}
 {\tan\beta=\tan\beta_1+2\sqrt{\tan\beta_1\,\tan\beta_2}},\quad \tan\beta>\tan\theta.
\end{equation}
Moreover, if $\alpha=\pi/2$, then $$\beta=\beta_2-\beta_1=\theta=\arctan\left(\frac{1}{m_{\infty}(-0)}\right).$$
\end{theorem}
\begin{proof}
We note first that the conditions of our theorem imply the following $\tan\alpha_*\in\left(\frac{1}{m_{\infty}(-0)},+\infty\right)$. Thus, according to \cite[Theorem 6.4]{BT18} part 2(c), a ($*$)-extension $\bA_{\tan\alpha,i}$ is $\beta$-sectorial for some $\beta\in(0,\pi/2)$. Then we can apply Theorem \ref{t-15} to confirm that $(-m_\alpha(z))\in S^{\beta_1,\beta_2}$, where $\beta_1$ and $\beta_2$ are described by \eqref{e-46} and \eqref{e-47}.
The first part of formula \eqref{e6-6} follows from \cite[Theorem 8]{BT-15} applied to the L-system $\Theta_{\tan\alpha, i}$ with $\mu=\tan\alpha$ (see also \cite[Theorem 9.8.7]{ABT}). Note that since $\bA_{\tan\alpha,i}$ is a $\beta$-sectorial extension of a $\theta$-sectorial operator $T_i$, then $\tan\beta\ge\tan\theta$ with equality possible only when $\mu=\tan\alpha=\infty$ (see \cite{ABT}, \cite{BT-15}). Since we chose $\alpha\in [\alpha_*,\pi/2)$, then $\tan\alpha\ne\infty$ and hence $\tan\beta>\tan\theta$ that confirms the second part of formula \eqref{e6-6}.

If we assume that $\alpha=\pi/2$, then our function $(-m_\alpha(z))=1/m_\infty(z)$ is realized with $\Theta_{\infty, i}$ (see Theorem \ref{t-7}) that preserves the angle of sectoriality of its main operator $T_i$ (see \cite[Theorem 6.4]{BT18} and Figure \ref{fig-1}). Therefore, $\beta=\theta$. If we combine this fact with $(-m_\alpha(z))\in S^{\beta_1,\beta_2}$ and apply Theorem \ref{t-15} we get that $\beta=\beta_2-\beta_1$. Finally, since $T_i$ is $\theta$-sectorial, formula \eqref{e10-45} yields $\tan\theta=\frac{1}{m_{\infty}(-0)}$.
 \end{proof}
Note that Theorem \ref{t-16} provides us with a value $\beta$ which serves as a universal angle of sectoriality for the entire indexed family of $(*)$-extensions $\bA_{\tan\alpha,i}$ of the form \eqref{e-64-sys} as depicted on Figure \ref{fig-2}. It is clearly shown on the figure that if $\alpha=\pi/2$, then $\tan\beta=\tan\theta$.

\section{Example}

We conclude this paper with a  simple illustration. Consider the differential expression with the Bessel potential
%\label{Bessel}
$$
l_\nu=-\frac{d^2}{dx^2}+\frac{\nu^2-1/4}{x^2},\;\; x\in [1,\infty)
$$
of order $\nu>0$ in the Hilbert space $\calH=L^2[1,\infty)$.
The minimal symmetric operator
\begin{equation}\label{ex-128}
 \left\{ \begin{array}{l}
 \dA\, y=-y^{\prime\prime}+\frac{\nu^2-1/4}{x^2}y \\
 y(1)=y^{\prime}(1)=0 \\
 \end{array} \right.
\end{equation}
 generated by this expression and boundary conditions has defect numbers $(1,1)$. %If $\nu\ge 0$ the operator $\dA$ is nonnegative.
% Consider also the operator
%\begin{equation}\label{ex-79}
% \left\{ \begin{array}{l}
% T_h\, y=-y^{\prime\prime}+\frac{\nu^2-1/4}{x^2}y \\
% y'(1)=h y(1). \\
% \end{array} \right.
%\end{equation}
Let $\nu=3/2$. It is known \cite{ABT} that in this case
$$
m_{\infty}(z)= -\frac{iz-\frac{3}{2}\sqrt{z}-\frac{3}{2}i}{\sqrt{z}+i}-\frac{1}{2}=\frac{\sqrt{z}-iz+i}{\sqrt{z}+i}=1-\frac{iz}{\sqrt{z}+i}
$$
and $m_{\infty}(-0)=1.$
The minimal symmetric operator then becomes
\begin{equation}\label{ex-e8-15}
 \left\{ \begin{array}{l}
 \dA\, y=-y^{\prime\prime}+\frac{2}{x^2}y \\
 y(1)=y^{\prime}(1)=0. \\
 \end{array} \right.
\end{equation}
The main operator $T_h$ of the form \eqref{131} is written for $h=i$ as %in \eqref{ex-79} as
\begin{equation}\label{ex-135}
 \left\{ \begin{array}{l}
 T_{i}\, y=-y^{\prime\prime}+\frac{2}{x^2}y \\
 y'(1)=i\, y(1) \\
 \end{array} \right.
\end{equation}
 will be shared by all the family of L-systems realizing functions $(-m_\alpha(z))$ described by \eqref{e-62-psi}-\eqref{e-59-LFT}. This operator is accretive and $\beta$-sectorial since $\RE h=0>-m_\infty(-0)=-1$ with the exact angle of sectoriality given by (see \eqref{e10-45})
\begin{equation}\label{e-ex-98}
\tan\beta=\frac{\IM h}{\RE h+m_{\infty}(-0)}=\frac{1}{0+1}=1\quad\textrm{ or}\quad \beta=\frac{\pi}{4}.
\end{equation}
A family of  L-systems $\Theta_{\tan\alpha, i}$ of the form \eqref{e-64-sys} that realizes functions $(-m_\alpha(z))$ described by \eqref{e-62-psi}--\eqref{e-61-Don} as
 \begin{equation}\label{e-ex-99}
  -m_\alpha(z)=\frac{({\sqrt{z}-iz+i})\cos\alpha+({\sqrt{z}+i})\sin\alpha}{({\sqrt{z}-iz+i})\sin\alpha-({\sqrt{z}+i})\cos\alpha},
 \end{equation}
  was constructed in \cite{BT18}.
 According to \cite[Theorem 6.3]{BT18} the L-systems  $\Theta_{\tan\alpha, i}$ in \eqref{e-64-sys} are accretive if
$$
 1=\frac{1}{m_\infty(-0)}\le\tan\alpha<+\infty.
$$
Using  part (2c) of \cite[Theorem 6.4]{BT18}, we get that the realizing L-system $\Theta_{\tan\alpha, i}$ in \eqref{e-64-sys} preserves the angle of sectoriality and becomes $\frac{\pi}{4}$-sectorial if $\mu=\tan\alpha=+\infty$ or $\alpha=\pi/2$. Therefore the L-system
\begin{equation}\label{e-ex-105-sys}
    \Theta_{\infty, i}= \begin{pmatrix} \bA_{\infty, i}&K_{\infty, i}&1\cr \calH_+ \subset
L_2[1,+\infty) \subset \calH_-& &\dC\cr \end{pmatrix},
\end{equation}
where
\begin{equation}\label{e-ex-106}
\begin{split}
&\bA_{\infty,i}\, y=-y^{\prime\prime}+\frac{2}{x^2}y-\,[y^{\prime}(1)-iy(1)]\,\delta(x-1), \\
&\bA^*_{\infty,i}\, y=-y^{\prime\prime}+\frac{2}{x^2}y-\,[y^{\prime}(1)+iy(1)]\,\delta(x-1),
\end{split}
\end{equation}
$K_{\infty, i}{c}=cg_{\infty, i}$, $(c\in \dC)$ and $g_{\infty, i}=\delta(x-1),$ realizes the function $-m_{\frac{\pi}{2}}(z)=1/m_\infty(z)$.
Also,
\begin{equation}\label{e-ex-108-VW}
    \begin{aligned}
    V_{\Theta_{\infty, i}}(z)&=-m_{\frac{\pi}{2}}(z)=\frac{1}{m_\infty(z)}=\frac{\sqrt{z}+i}{\sqrt{z}-iz+i}\\
     W_{\Theta_{\infty, i}}(z)&=(-e^{{\pi} i})\,\frac{m_\infty(z)-i}{m_\infty(z)+i}=\frac{(1-i)\sqrt{z}-iz+1+i}{(1+i)\sqrt{z}-iz-1+i}.
     \end{aligned}
\end{equation}
This L-system $\Theta_{\infty, i}$ is clearly accretive according to \cite[Theorem 6.2]{BT18} which is also independently confirmed by direct evaluation
$$
(\RE\bA_{\infty,i}\, y,y)=\|y'(x)\|^2_{L^2}+2\|y(x)/x\|^2_{L^2}\ge0.
$$
%The quasi-kernel $\hat A_{\infty,i}$ of $\RE\bA_{\infty,i}$ in \eqref{e-100-quasi} with $\alpha=\pi/2$ has boundary conditions $y'(1)=0$ and is not  the Krein-von Neumann extension of  $\dA$.
Moreover,  according to \cite[Theorem 6.4]{BT18} (see also \cite[Theorem 9.8.7]{ABT}) the L-system $\Theta_{\infty, i}$ is $\frac{\pi}{4}$-sectorial. %, i.e., $\bA$ is $\beta$-sectorial with $\beta=\frac{\pi}{4}$.
Taking into account that $(\IM \bA_{\infty,i}\, y,y)=|y(1)|^2,$ (see formula \eqref{146}) we obtain inequality \eqref{e8-29} with $\beta=\frac{\pi}{4}$, that is
$(\RE\bA_{\infty,i}\, y,y)\ge |(\IM\bA_{\infty,i}\, y,y)|,$ or
\begin{equation}\label{e8-26}
\|y'(x)\|^2_{L^2}+2\|y(x)/x\|^2_{L^2}\ge|y(1)|^2.
\end{equation}
%Note that inequality \eqref{e8-26} turns into equality on $y(x)=1/x$ which confirms that the angle of sectoriality of the L-system $\Theta_{\infty, i}$ in \eqref{e-ex-105-sys} is exact and equals $\beta=\pi/4$.
In addition, we have shown that  the $\beta$-sectorial  form $(T_i y,y)$ defined on  $\dom(T_i)$  can be extended to the $\beta$-sectorial form $(\bA_{\infty,i}\, y,y)$ defined on $\calH_+=\dom(\dA^*)$ (see \eqref{ex-e8-15}--\eqref{ex-135}) having the exact (for both forms) angle of sectoriality $\beta=\pi/4$. A general problem of extending sectorial sesquilinear forms to sectorial ones was mentioned by T.~Kato in \cite{Ka}. It can  be easily seen that function $-m_{\frac{\pi}{2}}(z)$ in \eqref{e-ex-108-VW} belongs to a sectorial class $S^{0,\frac{\pi}{4}}$ of Stieltjes functions.

%%%%%%%%%%%%%%%%%

\end{document}